\documentclass{amsart}
\usepackage{amssymb}
\usepackage{amsthm}
\usepackage{amsmath}
\usepackage{epic}
\usepackage{graphicx}
\usepackage{tabularx}
\usepackage{color}
\usepackage{hyperref}
\usepackage{url}
\usepackage[colorinlistoftodos]{todonotes}

\newtheorem{thm}{Theorem}[section]
\newtheorem{cor}[thm]{Corollary}
\newtheorem{lem}[thm]{Lemma}
\newtheorem{pro}[thm]{Proposition}

\theoremstyle{definition}

\newtheorem{exa}{Example}
\newtheorem{rem}{Remark}
\newtheorem{obs}{Observation}

\begin{document}

\title{Spectral radius and clique partitions of graphs}
\author[Jiang Zhou]{Jiang Zhou}
\author[Edwin R. van Dam]{Edwin R. van Dam}
\address[Jiang Zhou]{zhoujiang@hrbeu.edu.cn, College of Mathematical Sciences, Harbin Engineering University, Harbin 150001, PR China}
\address[Edwin R. van Dam]{Edwin.vanDam@uvt.nl, Department of Econometrics and O.R., Tilburg University, The Netherlands}

\begin{abstract}
We give lower bounds on the size and total size of clique partitions of a graph in terms of its spectral radius and minimum degree, and derive a spectral upper bound for the maximum number of edge-disjoint $t$-cliques. The extremal graphs attaining the bounds are exactly the block graphs of Steiner $2$-designs and the regular graphs with $K_t$-decompositions, respectively.
\end{abstract}



\maketitle

\section{Introduction}
A clique partition of a graph $G$ is a set of complete subgraphs of $G$ that partition the edge set of $G$. The clique partition number $\mbox{\rm cp}(G)$ is the minimum number of cliques in a clique partition of $G$. In 1948, De Bruijn and Erd\H{o}s \cite{Bruijn} proved that the clique partition number of the complete graph $K_n$ is at least $n$ (where the trivial partition of size $1$ is excluded) and characterized the clique partitions of size precisely $n$.

Bounds, asymptotic values, and explicit values of the clique partition number of some other dense graphs were given by (among others) Conlon, Fox, and Sudakov \cite{Conlon}, Erd\H{o}s, Faudree, and Ordman \cite{ErdosFaudree}, Orlin \cite{Orlin}, and Wallis \cite{Wallis}. An overview of early results can be found in the survey paper by Monson, Pullman, and Rees \cite{MPR} (from 1995). For complexity results and algorithms for the (NP-complete problem of) determination of the clique partition problem, we refer to the thesis of Cioab\u{a} \cite{Cioabathesis}. Nordhaus-Gaddum type results have been obtained by de Caen, Erd\H{o}s, Pullman, and Wormald \cite{Dom}.

Hoffman \cite{Hoffman} proved that for graphs without isolated vertices ${\mbox{\rm cp}(G)(\mbox{\rm cp}(G)+1)}/{2}$ is at least the number of eigenvalues of $G$ which are not $-1$, thus giving a bound involving the spectrum of the graph. He also proved that there exist graphs with the same spectrum having (very) different clique partition numbers (for other NP-complete problems with this property, see recent work of Etesami and Haemers \cite{EHaemers}). More recently, Cioab\u{a}, Elzinga, and Gregory \cite{Cioaba0} also published a spectral bound involving clique partitions.

Given a clique partition $\mathcal{Q}$ of $G$, its size is the number of elements of $\mathcal{Q}$. Here we are mainly interested in clique partitions in which every clique has size at most a given number $t$. The minimum size of such clique partitions is denoted by $\mbox{\rm cp}^{(t)}(G)$. For a clique partition $\mathcal{Q}$ of $G$, the total size of $\mathcal{Q}$ is the sum of the sizes of the elements in $\mathcal{Q}$. The minimum total size of a clique partition of $G$ is denoted by $\pi(G)$. When we restrict to cliques of size at most $t$, the corresponding minimum total size is denoted by $\pi_t(G)$. Some upper bounds on $\mbox{\rm cp}^{(3)}(G)$, $\pi(G)$ and $\pi_3(G)$ were given by Chung \cite{Chung}, Erd\H{o}s, Goodman, and P\'{o}sa \cite{Erdos}, Gy\H{o}ri and Kostochka \cite{Gyori}, Kahn \cite{Kahn}, and Kr\'{a}l', Lidick\'{y}, Martins, and Pehova \cite{Kral}.

We focus on the relation between the spectral radius $\rho(G)$, $\mbox{\rm cp}^{(t)}(G)$, and $\pi_t(G)$. In Section 2, we present definitions, lemmas, and known results. In Section 3, we give a spectral lower bound for $\mbox{\rm cp}^{(t)}(G)$. The extremal graphs are exactly the block graphs of Steiner $2$-designs. In Section 4, we give a spectral lower bound for $\pi_t(G)$, and derive a spectral upper bound for the maximum number of edge-disjoint $t$-cliques. The extremal graphs are exactly the regular graphs with $K_t$-decompositions (that is, a partition of the edge set into cliques of size $t$). In Section 5, we compare our bounds on the clique partition number to Hoffman's bound for several classes of graphs. In certain cases, our bounds turn out to be considerably better; see Table \ref{table1}.

\section{Preliminaries}

\subsection{Spectra of graphs}

Let $V(G)$ and $E(G)$ denote the vertex set and the edge set of a graph $G$, respectively. The \textit{adjacency matrix} $A(G)$ of $G$ is a $|V(G)|\times|V(G)|$ symmetric matrix with entries
\begin{eqnarray*}
(A(G))_{ij}=\begin{cases}1~~~~~~~~~~~~~~~~~~\mbox{if}~\{i,j\}\in E(G),\\
0~~~~~~~~~~~~~~~~~~\mbox{if}~\{i,j\}\notin E(G).\end{cases}
\end{eqnarray*}
Eigenvalues of $A(G)$ are called \textit{eigenvalues} of $G$. The largest eigenvalue of $G$ is called the \textit{spectral radius} of $G$, denoted by $\rho(G)$. When $G$ is connected, there are positive eigenvectors for the spectral radius, called Perron eigenvectors.
The following is a well-known result in spectral graph theory.
\begin{lem}\label{lem2.3}
\textup{\cite{CveRS}} Let $G$ be a connected graph with $n$ vertices. Then $\rho(G)\leq n-1$, with equality if and only if $G$ is the complete graph $K_n$.
\end{lem}

For a real symmetric matrix $A$ of order $n$, let $\lambda_1(A)\geq\cdots\geq\lambda_n(A)$ denote all eigenvalues of $A$.
\begin{lem}\label{lem2.4}
\textup{\cite{Horn}} Let $A$ and $B$ be $n\times n$ real symmetric matrices. Then
\begin{eqnarray*}
\lambda_i(A+B)\leq\lambda_j(A)+\lambda_{i+1-j}(B)~(1\leq j\leq i\leq n),\\
\lambda_i(A+B)\geq\lambda_j(A)+\lambda_{i+n-j}(B)~(1\leq i\leq j\leq n).
\end{eqnarray*}
In either of these inequalities, equality holds if and only if there exists a vector that is an eigenvector to each of the three involved eigenvalues.
\end{lem}
The following is a well-known result in nonnegative matrix theory.
\begin{lem}\label{lem2.5}
\textup{\cite{CveRS}} Let $A$ be an irreducible symmetric nonnegative matrix with spectral radius $\lambda$ and maximum row sum $r$. Then $\lambda\leq r$, with equality if and only if all row sums of $A$ are equal.
\end{lem}

\subsection{Clique partitions}

A \textit{clique} in $G$ is a vertex set $Q\subseteq V(G)$ such that the subgraph induced on $Q$ is complete. A clique with $t$ vertices is called a \textit{$t$-clique}. A \textit{clique partition} \cite{ZhouBu2} of $G$ is a set $\mathcal{Q}=\{Q_1,\ldots,Q_v\}$ of cliques such that each edge of $G$ belongs to exactly one clique $Q_i$ of $\mathcal{Q}$, that is, any two distinct cliques
in $\mathcal{Q}$ have at most one common vertex. We say that $\mathcal{Q}$ is a \textit{$K_t$-decomposition} of $G$ if $|Q_i|=t$ for $i=1,\ldots,v$. The \textit{clique partition number} \cite{ZhouBu1} $\mbox{\rm cp}(G)$ is the smallest size of a clique partition of $G$.

Let $\Omega(\mathcal{Q})$ denote the \textit{clique graph} with vertex set $\mathcal{Q}=\{Q_1,\ldots,Q_v\}$ and edge set $\{\{Q_i,Q_j\}:i\neq j,Q_i\cap Q_j\neq\emptyset\}$. The \textit{$\mathcal{Q}$-degree} $d_u^\mathcal{Q}$ of a vertex $u$ in $G$ is the number of cliques in $\mathcal{Q}$ containing $u$.

The vertex-clique incidence matrix $B_\mathcal{Q}$ is a $|V(G)|\times v$ matrix with entries
\begin{eqnarray*}
(B_\mathcal{Q})_{ij}=\begin{cases}1~~~~~~~~~~~~~~~~~~\mbox{if}~i\in V(G),i\in Q_j,\\
0~~~~~~~~~~~~~~~~~~\mbox{if}~i\in V(G),i\notin Q_j.\end{cases}
\end{eqnarray*}
It follows easily that
\begin{equation}\label{BBT}
B_\mathcal{Q} B_\mathcal{Q}^\top=A(G)+D\mbox{~~and~~}B_\mathcal{Q}^\top B_\mathcal{Q}=A(\Omega(\mathcal{Q}))+E,
\end{equation}
where $D$ is the diagonal matrix satisfying $(D)_{uu}=d_u^\mathcal{Q}$ for each $u\in V(G)$, and $E$ is the diagonal matrix satisfying $(E)_{ii}=|Q_i|$ ($i=1,\ldots,v$).

\begin{exa}\label{ex1}
The graph $G$ in Figure \ref{clique} has a clique partition $\mathcal{Q}=\{Q_1,Q_2,Q_3\}$, where $Q_1=\{1,2,3\},Q_2=\{2,4,5\},Q_3=\{3,5,6\}$. The $\mathcal{Q}$-degrees are  $d_1^\mathcal{Q}=1$, $d_2^\mathcal{Q}=2$, $d_3^\mathcal{Q}=2$, $d_4^\mathcal{Q}=1$, $d_5^\mathcal{Q}=2$, $d_6^\mathcal{Q}=1$. The clique partition number of $G$ is $\mbox{\rm cp}(G)=3$.
\begin{figure}
\centering
\includegraphics*[height=4.6cm]{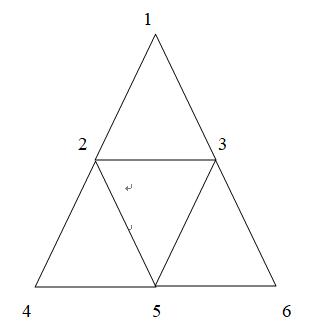}
\caption{\small Graph $G$ of Example \ref{ex1}}
\label{clique}
\end{figure}
\end{exa}

\subsection{Strongly regular graphs and block designs}

A \textit{strongly regular graph} \cite{Brouwer} with parameters $(n,d,\lambda,\mu)$ is a (non-complete) $d$-regular graph on $n$ vertices such that for every pair of adjacent vertices there are $\lambda$ vertices adjacent to both, and for every pair of non-adjacent vertices there are $\mu$ vertices adjacent to both.

A $2$-$(v,k,\lambda)$ design \cite{Brouwer} is a set of $v$ points, with a set of blocks (subsets of points) of size $k$ such that any two points meet in exactly $\lambda$ blocks. A $2$-$(v,k,1)$ design is also called a Steiner $2$-design. The \textit{block graph} of a Steiner $2$-design $\mathcal{D}$ is the graph whose vertex set is the set of blocks of $\mathcal{D}$, and two blocks adjacent when they meet in one point. It is known that block graphs of Steiner $2$-designs are strongly regular or complete \cite[\S 9.1.10]{Brouwer}.

As shown in the following observations, there is a one-to-one correspondence between special $K_t$-decompositions of regular graphs and Steiner $2$-designs.
\begin{obs}\label{obs1}
Let $\mathcal{Q}=\{Q_1,\ldots,Q_v\}$ be a $K_t$-decomposition of a $d$-regular graph $G$ such that $|Q_i\cap Q_j|=1$ for any $i\neq j$; in particular, the clique graph is complete.
For a vertex $u\in V(G)$, let $S_u=\{Q_i:u\in Q_i\}$ denote the set of $t$-cliques containing $u$. Let $\mathcal{D}$ be the design with point set $\mathcal{Q}$ and block set $\{S_u:u\in V(G)\}$. Then $\mathcal{D}$ is a $2$-$(v,\frac{d}{t-1},1)$ design and $G$ is the block graph of $\mathcal{D}$.
\end{obs}

\begin{obs}\label{obs2}
If $G$ is the block graph of a $2$-$(v,k,1)$ design $\mathcal{D}$, then the set of blocks incident to a given point $i$ of $\mathcal{D}$ corresponds to a $t$-clique $Q_i$ of $G$, where $t=k|V(G)|/v$. These cliques form a $K_t$-decomposition $\mathcal{Q}=\{Q_1,\ldots,Q_v\}$ of $G$, and $|Q_i\cap Q_j|=1$ for any $i\neq j$. Note also that a standard counting argument implies that $v-1=t(k-1)$.
\end{obs}

\subsection{The size of clique partitions}

Hoffman \cite{Hoffman} gave the following two spectral lower bounds for $\mbox{\rm cp}(G)$.
\begin{pro}\label{thm1.2}
\textup{\cite{Hoffman}} Let $G$ be a graph with no isolated vertices, and let $q$ be the number of eigenvalues of $G$ which are not $-1$. Then
\begin{eqnarray*}
\frac{\mbox{\rm cp}(G)(\mbox{\rm cp}(G)+1)}{2}\geq q,
\end{eqnarray*}
\end{pro}
We note that this bound is equivalent to
\begin{eqnarray*}
\mbox{\rm cp}(G)\geq\frac{-1+\sqrt{8q+1}}{2}.
\end{eqnarray*}

\begin{pro}\label{thm1.3}
\textup{\cite{Hoffman}} Let $\lambda_{min}(G)$ be the minimum eigenvalue of a graph $G$. Then $\mbox{\rm cp}(G)\geq-\lambda_{min}(G)$.
\end{pro}
Erd\H{o}s, Goodman, and P\'{o}sa \cite{Erdos} proved that every graph on $n$ vertices can be decomposed into at most $\lfloor n^2/4\rfloor$ complete graphs. In fact, they proved Proposition \ref{thm1.7} stated below.
Let $\mbox{\rm cp}^{(t)}(G)$ denote the minimum integer $m$ such that $G$ has a clique partition $\mathcal{Q}$ of size $|\mathcal{Q}|=m$ with every clique in $\mathcal{Q}$ having at most $t$ vertices. Clearly, $\mbox{\rm cp}^{(t)}(G)=\mbox{\rm cp}(G)$ when $t$ is at least the clique number of $G$. Erd\H{o}s et al. \cite{Erdos} obtained the following.
\begin{pro}\label{thm1.7}
\textup{\cite{Erdos}} Let $G$ be a graph on $n$ vertices. Then $\mbox{\rm cp}^{(3)}(G)\leq\lfloor n^2/4\rfloor$.
\end{pro}

\subsection{The total size of clique partitions}

For a clique partition $\mathcal{Q}=\{Q_1,\ldots,Q_v\}$ of $G$, the \textit{total size} of $\mathcal{Q}$ is defined as $s(\mathcal{Q})=\sum_{i=1}^v|Q_i|$.
We let $\mathcal{C}(G)$ denote the set of all clique partitions of $G$, and define $\pi(G)=\min_{\mathcal{Q}\in\mathcal{C}(G)}s(\mathcal{Q})$ as the minimum total size over all clique partitions of $G$.

Chung \cite{Chung}, Gy\H{o}ri and Kostochka \cite{Gyori}, and Kahn \cite{Kahn}, all independently, obtained the following upper bound for $\pi(G)$, which is a generalization of the (above) bound on the clique partition number by Erd\H{o}s et al. \cite{Erdos}.

\begin{pro}
\textup{\cite{Chung,Gyori,Kahn}} Let $G$ be a graph on $n$ vertices. Then $\pi(G)\leq n^2/2$.
\end{pro}

Let $\pi_t(G)$ denote the minimum integer $m$ such that $G$ has a clique partition $\mathcal{Q}$ of total size $s(\mathcal{Q})=m$ with every clique in $\mathcal{Q}$ having at most $t$ vertices. Kr\'{a}l' et al. \cite{Kral} proved the following result.
\begin{pro}\label{thm4.6}
\textup{\cite{Kral}} Let $G$ be a graph on $n$ vertices. Then $\pi_3(G)\leq(1/2+o(1))n^2$.
\end{pro}

\section{Spectral bounds for the size of clique partitions}\label{sec:size}

We now present our results on the size of clique partitions.

\begin{thm}\label{thm3.1}
Let $t \geq 2$ and let $G$ be a connected graph with minimum degree $\delta(G)$ and spectral radius $\rho(G)$. Then
\begin{eqnarray*}
\mbox{\rm cp}^{(t)}(G)\geq\rho(G)-t+1+\left\lceil\frac{\delta(G)}{t-1}\right\rceil,
\end{eqnarray*}
with equality if and only if $G$ is the block graph of a $2$-$(v,k,1)$ design, where $k=\frac{\delta(G)}{t-1}$ and $v=k|V(G)|/t$.
\end{thm}

\begin{proof}
Suppose that $\mathcal{Q}=\{Q_1,\ldots,Q_v\}$ is a clique partition of $G$ such that $v=\mbox{\rm cp}^{(t)}(G)$ and $|Q_i|\leq t$ for $i=1,\ldots,v$.  By \eqref{BBT}, Lemmas \ref{lem2.4} and \ref{lem2.3}, we have
\begin{eqnarray*}
\rho(G)+\left\lceil\frac{\delta(G)}{t-1}\right\rceil&\leq&\rho(G)+\min_{u\in V(G)}d_u^\mathcal{Q}\leq\lambda_1(B_\mathcal{Q} B_\mathcal{Q}^\top)=\lambda_1(B_\mathcal{Q}^\top B_\mathcal{Q})\\
&\leq&\rho(\Omega(\mathcal{Q}))+\max_{1\leq i\leq v}|Q_i|\leq v-1+t,
\end{eqnarray*}
so $$\mbox{\rm cp}^{(t)}(G)=v\geq\rho(G)-t+1+\left\lceil\frac{\delta(G)}{t-1}\right\rceil.$$
In case of equality, we first of all have that $\rho(\Omega(\mathcal{Q}))=v-1$, so the clique graph $\Omega(\mathcal{Q})$ is the complete graph $K_v$ by Lemma \ref{lem2.3}. Moreover, by Lemma \ref{lem2.4}, a positive Perron eigenvector of $\Omega(\mathcal{Q})$ is also an eigenvector for the diagonal matrix $E$ in \eqref{BBT}, which implies that $|Q_1|=\cdots=|Q_v|=t$ (since the corresponding eigenvalue is $\max_{1\leq i\leq v}|Q_i|=t$).
Similarly, Lemma \ref{lem2.4} implies that a (again positive) Perron eigenvector of $A(G)$ is an eigenvector for the smallest eigenvalue $\left\lceil\frac{\delta(G)}{t-1}\right\rceil =\min_{u\in V(G)}d_u^\mathcal{Q}$ of the diagonal matrix $D$ in \eqref{BBT}, hence $d_u^\mathcal{Q}=\left\lceil\frac{\delta(G)}{t-1}\right\rceil $ for all $u$. In particular, every vertex is in the same number of cliques $Q_i$, which then implies that $G$ is regular with degree $(t-1)\left\lceil\frac{\delta(G)}{t-1}\right\rceil $. Thus, $t-1$ divides $\delta(G)$, and $\mathcal{Q}$ is a $K_t$-decomposition of $G$ with a complete clique graph. By Observation \ref{obs1}, it follows that $G$ is the block graph of a $2$-$(v,k,1)$ design, where $k=\frac{\delta(G)}{t-1}$ and $v=k|V(G)|/t$.

On the other hand, it follows easily from Observation \ref{obs2} that the bound is tight for the block graph of a $2$-$(v,k,1)$ design.
\end{proof}

\begin{rem}
The extremal graphs $G$ attaining the lower bound in Theorem \ref{thm3.1}, the block graphs of $2$-$(v,k,1)$ designs, can be classified as follows using the eigenvalues of $G$. Indeed, from \eqref{BBT} it follows that $G$ has eigenvalues $\delta(G)$ (with multiplicity $1$), $t-1-k$ (multiplicity $v-1$) and $-k$ (multiplicity $|V(G)|-v$). We can now distinguish the followings cases depending on the vanishing of the latter two multiplicities.

(I) $v=1$: In this case, $G$ is complete; $G=K_t$ and $\mbox{\rm cp}^{(t)}(G)=1$.

(II) $v=|V(G)|$: In this case, $G$ is also complete, and it follows from Observation \ref{obs2} that $t=k$ and $v=t^2-t+1$. The corresponding design $\mathcal{D}$ is a projective plane of order $t-1$; so $G=K_{t^2-t+1}$, $\mbox{\rm cp}^{(t)}(G)=t^2-t+1$, and the $t$-cliques of an extremal clique partition are the lines of a projective plane (note however that this is the dual projective plane of $\mathcal{D}$). Note that projective planes are known to exist if the order is a prime power.

(III) $v\neq1,|V(G)|$: In this case, $G$ is a strongly regular graph with a $K_t$-decomposition and distinct eigenvalues $\delta(G)$, $t-1-k$, and $-k$.
\end{rem}
Case (II) is one of the two cases for a nontrivial extremal clique partition of the complete graph (case (I) is trivial), as characterized by De Bruijn and Erd\H{o}s \cite{Bruijn}; the other one is a partition of $K_v$ into a $(v-1)$-clique and the $v-1$ edges through the missing vertex.

The following are two examples for case (III).
\begin{exa}\label{triangular}
Let $v=t+1$. The pairs of a set of size $v$ form the blocks of a (trivial) $2$-$(v,2,1)$ design. The block graph of this design is the triangular graph $T(v)$ (the line graph of $K_v$). Let $Q_i$ be the set of pairs $\{i,j\}$, with $i\neq j$. These sets $Q_i$ with $i=1,\dots,v$ form an extremal clique partition; i.e., $\mbox{\rm cp}^{(t)}(G)=t+1$. Note that the sets $Q_i$ are the cliques of maximal size, so also $\mbox{\rm cp}(G)=t+1$.
\end{exa}

\begin{exa}\label{affine}
Let $k=t-1$. A $2$-$(k^2,k,1)$ design is equivalent to an affine plane of order $k$. Such affine planes are known to exist when $k$ is a prime power. The block graph of an affine plane is a complete $(k+1)$-partite graph $K_{(k+1) \times k}$. Let $Q_x$ be the set of $k+1$ lines through point $x$ of the affine plane. The set of $k^2$ cliques $Q_x$ forms an extremal clique partition, i.e., $\mbox{\rm cp}^{(t)}(G)=(t-1)^2$. Also here, the cliques have maximal size, so $\mbox{\rm cp}(G)=(t-1)^2$ as well.
\end{exa}

The above examples involve maximal cliques. In general, we obtain the following spectral bound for $\mbox{\rm cp}(G)$ from Theorem \ref{thm3.1} by involving the clique number $\omega(G)$ of $G$ and observing that $\mbox{\rm cp}(G)= \mbox{\rm cp}^{(\omega(G))}(G)$.
\begin{cor}\label{cor3.4}
Let $G$ be a connected graph with minimum degree $\delta(G)$, spectral radius $\rho(G)$, and clique number $\omega(G)$. Then
\begin{eqnarray*}
\mbox{\rm cp}(G)\geq\rho(G)-\omega(G)+1+\left\lceil\frac{\delta(G)}{\omega(G)-1}\right\rceil,
\end{eqnarray*}
with equality if and only if $G$ is the block graph of a $2$-$(v,k,1)$ design, where $k=\frac{\delta(G)}{\omega(G)-1}$ and $v=k|V(G)|/\omega(G)$.
\end{cor}

\section{Spectral bounds for the total size of clique partitions}

We next consider the minimum total size of clique partitions.
\begin{thm}\label{thm3.7}
Let $t \geq 2$ and let $G$ be a connected graph with $n$ vertices, minimum degree $\delta(G)$, and spectral radius $\rho(G)$. Then
\begin{equation}\label{eq:totalsize}
\pi_t(G)\geq\rho(G)+(n-t+1)\left\lceil\frac{\delta(G)}{t-1}\right\rceil,
\end{equation}
with equality if and only if $G$ is regular with a $K_t$-decomposition.
\end{thm}

\begin{proof}
Let $\mathcal{Q}=\{Q_1,\ldots,Q_v\}$ be a clique partition of $G$ such that $s(\mathcal{Q})=\sum_{i=1}^v|Q_i|=\pi_t(G)$ and $|Q_i|\leq t$ for $i=1,\ldots,v$.
As in the proof of Theorem \ref{thm3.1}, we have that
\begin{equation}\label{eq2}
\rho(G)+\left\lceil\frac{\delta(G)}{t-1}\right\rceil\leq\rho(G)+\min_{u\in V(G)}d_u^\mathcal{Q}\leq\lambda_1(B_\mathcal{Q} B_\mathcal{Q}^\top)=\lambda_1(B_\mathcal{Q}^\top B_\mathcal{Q}).
\end{equation}
We are going to bound the right hand side of this equation by the maximal row sum (say of row $j$) of $B_\mathcal{Q}^\top B_\mathcal{Q}$, which equals
$$\sum_{i=1}^v (B_\mathcal{Q}^\top B_\mathcal{Q})_{ji}=\sum_{i=1}^v \sum_{u\in V(G)} (B_\mathcal{Q})_{uj}(B_\mathcal{Q})_{ui}=\sum_{u \in Q_j}d_u^\mathcal{Q}.$$

Note also that $\pi_t(G)=\sum_{i=1}^v|Q_i|=\sum_{u\in V(G)}d_u^\mathcal{Q}$. By Lemma \ref{lem2.5}, we then have that
\begin{eqnarray*}
\lambda_1(B_\mathcal{Q}^\top B_\mathcal{Q})&\leq&\sum_{u\in Q_j}d_u^\mathcal{Q}=\pi_t(G)-\sum_{u\notin Q_j}d_u^\mathcal{Q}\leq\pi_t(G)-(n-|Q_j|)\left\lceil\frac{\delta(G)}{t-1}\right\rceil\\
&\leq&\pi_t(G)-(n-t)\left\lceil\frac{\delta(G)}{t-1}\right\rceil.
\end{eqnarray*}
Together with \eqref{eq2}, we thus obtain the bound \eqref{eq:totalsize}.

In case of equality, we have that $d_u^\mathcal{Q}=\left\lceil\frac{\delta(G)}{t-1}\right\rceil$ for each $u\in V(G)$, just as in the proof of Theorem \ref{thm3.1}. Moreover, Lemma \ref{lem2.5} implies that $B_\mathcal{Q}^\top B_\mathcal{Q}$ has constant row sums, from which it then follows that all cliques $Q_i$ have the same size, which must be $t$. Thus, $G$ is regular with a
$K_t$-decomposition. On the other hand it is easy to check that if $G$ is regular with a
$K_t$-decomposition, then the bound is attained.
\end{proof}

Let $k_t(G)$ denote the maximum number of edge-disjoint $t$-cliques in $G$. By using Theorem \ref{thm3.7}, we can derive the following spectral bound for $k_t(G)$.
\begin{pro}\label{thm4.2}
Let $t>2$ and let $G$ be a connected graph with $n$ vertices, $m$ edges, minimum degree $\delta(G)$, and spectral radius $\rho(G)$. Then
\begin{equation}\label{eq:edgedisjoint}
k_t(G)\leq\frac{2m-\rho(G)-(n-t+1)\left\lceil\frac{\delta(G)}{t-1}\right\rceil}{t(t-2)},
\end{equation}
with equality if and only if $G$ is regular with a $K_t$-decomposition.
\end{pro}

\begin{proof}
Consider a set of $k_t(G)$ edge-disjoint $t$-cliques in $G$. Together with the  $m-\frac{t(t-1)}{2}k_t(G)$ edges (seen as cliques of size $2$) not in any of these $t$-cliques, they form a clique partition $\mathcal{Q}$ of $G$. By Theorem \ref{thm3.7}, we thus have that
\begin{equation}\label{eq:derivation}
tk_t(G)+2m-t(t-1)k_t(G)\geq\pi_t(G)\geq\rho(G)+(n-t+1)\left\lceil\frac{\delta(G)}{t-1}\right\rceil,
\end{equation}
and the bound \eqref{eq:edgedisjoint} follows.
If equality holds, then equality should hold in \eqref{eq:totalsize} of Theorem \ref{thm3.7}, so $G$ is regular with a $K_t$-decomposition.

On the other hand, if $G$ is regular with a $K_t$-decomposition of size $v$, then $tv=\pi_t(G)$ and no edges are left, from which it follows that we have equality in \eqref{eq:derivation} and hence in \eqref{eq:edgedisjoint}, with $v=k_t(G)$.
\end{proof}

Just like in Section \ref{sec:size}, by involving the clique number $\omega(G)$, we obtain spectral bounds from Theorem \ref{thm3.7}, for $\pi(G)$ and $\mbox{\rm cp}(G)$. Here we use in addition that $\pi(G) \leq \omega(G)\mbox{\rm cp}(G)$, which is easy to show by considering a clique partition with size $\mbox{\rm cp}(G)$.
\begin{cor}\label{cor4.3}
Let $G$ be a connected graph with $n$ vertices, minimum degree $\delta(G)$, spectral radius $\rho(G)$, and clique number $\omega(G)$. Then
\begin{eqnarray*}
\omega(G)\mbox{\rm cp}(G)\geq\pi(G)\geq\rho(G)+(n-\omega(G)+1)\left\lceil\frac{\delta(G)}{\omega(G)-1}\right\rceil.
\end{eqnarray*}
with equality in the right inequality if and only if $G$ is regular with a $K_{\omega(G)}$-decomposition.
\end{cor}

Note that equality in the left inequality does not imply that $G$ is regular; see Example \ref{ex1}, which has $\omega(G)=3$, $\mbox{\rm cp}(G)=3$, and $\pi(G)=9$.

\section{Comparison to Hoffman's bounds}

\vspace{3mm}
\begin{table}
\begin{center}
\begin{tabular}{|c|c|c|c|c|c|}
  \hline
   $G$  & Cor.~\ref{cor3.4} & Cor.~\ref{cor4.3} & Prop.~\ref{thm1.2} & Prop.~\ref{thm1.3} & \mbox{\rm cp}($G$)\\
 \hline
   &  &  &  & &\\
  $K_{p \times a}$ & $p(a-1)+1$ & $a^2$  & $\frac{-1+\sqrt{8pa+1}}{2}$ & $a$ & unknown\\
   &  &  &  & &\\
 $\overline{C_{2s+1}}$ & $s+1$ & $\frac{4s+2}{s}$ & $\frac{-1+\sqrt{16s+9}}{2}$ & $2\cos\frac{2\pi}{2s+1}+1$ & unknown\\
   &  &   &  & &\\
 $T(v)$ & $v$ & $v$  & $v-1$ & $2$ & $v$\\
   &  &  &  & &\\
 $F_v$ & $\frac{-1+\sqrt{8v+1}}{2}$ & $\frac{4v-1+\sqrt{8v+1}}{6}$  & $\frac{-1+\sqrt{8v+9}}{2}$ & $\frac{-1+\sqrt{8v+1}}{2}$ & $v$\\
   &  &  &  & &\\
  \hline
\end{tabular}
\caption{Spectral lower bounds for the clique partition number}\label{table1}
\end{center}
\end{table}

In Table \ref{table1}, we list the spectral lower bounds for the clique partition number from Propositions \ref{thm1.2} and \ref{thm1.3} (by Hoffman \cite{Hoffman}) and our Corollaries \ref{cor3.4} and \ref{cor4.3} for several classes of graphs. The last column gives the exact value of the clique partition number. We consider the following graphs.

(i) The complete $p$-partite graph $K_{p\times a}$ ($p>1,a>1$), which has distinct eigenvalues $(p-1)a,0,-a$ and clique number $p$. The clique partition number is at least $a^2$, with equality if and only if there is an orthogonal array OA$(a,p)$ (or equivalently, a so-called transversal design TD$(p,a)$ \cite{Brouwer}). This is an array with $p$ rows and $a^2$ columns containing $a$ symbols such that for any two rows, each of the $a^2$ pairs of symbols occurs precisely once \cite{Brouwer}; each of the columns corresponds to a clique in the $K_{a}$-decomposition. Note that an orthogonal array OA$(a,a+1)$ is equivalent to an affine plane of order $a$. See also Example \ref{affine}.

(ii) The complement $\overline{C_{2s+1}}$ ($s>1$) of an odd cycle, which has eigenvalues $2s-2$ and $-2\cos\frac{2\pi j}{2s+1}-1$ ($j=1,\ldots,2s$) and clique number $s$. For information (and references) on the clique partition number of (among others) the complement of the cycle, we refer to the thesis of Cavers \cite{Cavers}, who showed that the clique partition number of the complement of $C_n$ is $O(n \log\log n)$, while also conjecturing that it is of order $n$.

(iii) The triangular graph $T(v)$ ($v>3$), which has distinct eigenvalues $2v-4,v-4,-2$ and clique number $v-1$. See also Example \ref{triangular}.

(iv) The friendship graph $F_v$ ($v\geq2$) consisting of $v$ edge-disjoint triangles that meet in one vertex. It has distinct eigenvalues $\frac{1\pm\sqrt{8v+1}}{2},\pm1$, with $-1$ having  multiplicity $v$ \cite{Cioaba} and clique number $3$.

\vspace{3mm}
\noindent
\textbf{Acknowledgements.}

\vspace{3mm}
The authors would like to thank the reviewers for giving valuable suggestions. This work is supported by the National Natural Science Foundation of China (No. 11801115 and No. 12071097) and the Fundamental Research Funds for the Central Universities.


\begin{thebibliography}{}
\bibitem{Brouwer}A.E. Brouwer, W.H. Haemers, Spectra of Graphs, Springer, New York, 2012.
\bibitem{Bruijn}N.G. de Bruijn, P. Erd\H{o}s, On a combinatorial problem, Indag. Math. 10 (1948) 421-423.
\bibitem{Dom}D. de Caen, P. Erd\H{o}s, N.J. Pullmann, N.C. Wormald, Extremal clique coverings of complementary graphs, Combinatorica 6 (1986) 309-314.
\bibitem{Cavers} M.S. Cavers, Clique partitions and coverings of graphs, M.Sc. thesis, University of Waterloo, Canada, 2005.
\bibitem{Chung}F.R.K. Chung, On the decomposition of graphs, SIAM J. Algebraic Discrete Methods 2 (1981) 1-12.
\bibitem{Cioabathesis}S.M. Cioab\u{a}, The NP-completeness of some edge-partitioning problems, M.Sc. thesis, Queen's University, Canada, 2002.
\bibitem{Cioaba0}S.M. Cioab\u{a}, R.J. Elzinga, D.A. Gregory, Some observations on the smallest adjacency eigenvalue of a graph, Discuss. Math. Graph Theory 40 (2020) 467-493.
\bibitem{Cioaba}S.M. Cioab\u{a}, W.H. Haemers, J. Vermette, W. Wong, The graphs with all but two eigenvalues equal to $\pm1$, J. Algebr. Comb. 41 (2015) 887-897.
\bibitem{Conlon}D. Conlon, J. Fox, B. Sudakov, Short proofs of some extremal results, Combin. Probab. Comput. 23 (2014) 8-28.
\bibitem{CveRS}D. Cvetkovi\'{c}, P. Rowlinson, S. Simi\'{c}, An Introduction to the Theory of Graph Spectra, Cambridge University Press, Cambridge, 2010.
\bibitem{ErdosFaudree}P. Erd\H{o}s, R. Faudree, E.T. Ordman, Clique partitions and clique coverings, Discrete Math. 72 (1988) 93-101.
\bibitem{Erdos}P. Erd\H{o}s, A.W. Goodman, L. P\'{o}sa, The representation of a graph by set intersections, Canad. J. Math. 18 (1966) 106-112.
\bibitem{EHaemers}O. Etesami, W.H. Haemers, On NP-hard graph properties characterized by the spectrum, Discrete Appl. Math. 285 (2020) 526-529.
\bibitem{Gyori}E. Gy\H{o}ri, A.V. Kostochka, On a problem of G. O. H. Katona and T. Tarj\'{a}n, Acta Math. Acad. Sci. Hungar. 34 (1979) 321-327.
\bibitem{Hoffman}A.J. Hoffman, Eigenvalues and partitionings of the edges of a graph, Linear Algebra Appl. 5 (1972) 137-146.
\bibitem{Horn}R.A. Horn, N.H. Rhee, W. So, Eigenvalue inequalities and equalities, Linear Algebra Appl. 270 (1998) 29-44.
\bibitem{Kahn}J. Kahn, Proof of a conjecture of Katona and Tarj\'{a}n, Period. Math. Hungar. 12 (1981) 81-82.
\bibitem{Kral}D. Kr\'{a}l', B. Lidick\'{y}, T.L. Martins, Y. Pehova, Decomposing graphs into edges and triangles, Combin. Probab. Comput. 28 (2019) 465-472.
\bibitem{MPR}S.D. Monson, N.J. Pullman, R. Rees, A survey of clique and biclique coverings and factorizations of $(0, 1)$-matrices, Bull. Inst. Combin. Appl. 14 (1995), 17-86.
\bibitem{Orlin}J. Orlin, Contentment in graph theory: covering graphs with cliques, Indag. Math. 39 (1977) 406-424.
\bibitem{Wallis}W.D. Wallis, Asymptotic values of clique partition numbers, Combinatorica 2 (1982) 99-101.
\bibitem{ZhouBu1}J. Zhou, C. Bu, Eigenvalues and clique partitions of graphs, Adv. Appl. Math. 129 (2021) 102220.
\bibitem{ZhouBu2}J. Zhou, C. Bu, The enumeration of spanning tree of weighted graphs, J. Algebr. Comb., DOI: 10.1007/s10801-020-00969-w.
\end{thebibliography}
\end{document}